\documentclass{article}[11pt,a4paper]

\usepackage[english]{babel}
\usepackage{amsmath}
\usepackage{amsfonts}
\usepackage{amsthm}
\usepackage{amscd}
\usepackage{amssymb}

\usepackage[utf8]{inputenc}
\usepackage{enumerate}
\usepackage{graphicx}
\usepackage[margin=3.1cm]{geometry}

\newtheoremstyle{mystyle}{}{}{\slshape}{2pt}{\scshape}{.}{ }{}
\newtheoremstyle{etapestyle}{}{}{\itshape}{2em}{\sffamily}{:}{ }{\thmname{#1}}
\newtheoremstyle{definitionstyle}{}{}{}{2pt}{\bfseries}{.}{ }{}

\newtheorem{thm}{Theorem}[section]

\newtheorem{prop}[thm]{Proposition}
\newtheorem{defi}[thm]{Definition}

\newtheorem{lemma}[thm]{Lemma}

\newtheorem{question}[thm]{Question}

\theoremstyle{mystyle}

\newtheorem{notation}[thm]{Notation}
\theoremstyle{remark}
\newtheorem{rem}[thm]{Remark}

\theoremstyle{etapestyle}

\theoremstyle{definitionstyle}

\newcommand{\op}{\operatorname}
\newcommand{\mc}{\mathbb{C}}
\newcommand{\mz}{\mathbb{Z}}

\newcommand{\mq}{\mathbb{Q}}

\newcommand{\mr}{\mathbb{R}}

\newcommand{\plan}{\mathcal{P}}

\newcommand{\si}{\sigma}

\newcommand{\rw}{\rightarrow}

\newcommand{\me}{\mathcal{E}}
\newcommand{\md}{\mathcal{D}}
\newcommand{\ml}{\mathcal{L}}
\newcommand{\mo}{\mathcal{O}}

\newcommand{\kb}{\bar{k}}
\newcommand{\open}{\Omega}
\newcommand{\akx}{\op{Aut}_K(X)}
\newcommand{\ep}{\si}

\newcommand{\g}{\Gamma}

\newcommand{\cat}{\mathcal{C}}

\newcommand{\gmk}{\mathbb{G}_{m,k}}

\newcommand{\zzz}{\mathbb{Z}/3\mathbb{Z}}

\newcommand{\z}{\mathbb{Z}}

\newcommand{\Xokb}{X_{0,\kb}}
\newcommand{\Gkb}{G_{\kb}}

\newcommand{\xokb}{x_{0,\kb}}

\newcommand{\kbXo}{\mathcal{K}}

\newcommand{\mf}{\mathcal{F}}

\newcommand{\nkb}{\mathcal{V}}

\newcommand{\mx}{\mathcal{X}}

\newcommand{\ka}{\mathcal{K}}

\newcommand{\gsi}{g_{\si}}
\newcommand{\gsii}{g_{\si}^{-1}}
\newcommand{\gsiii}{g_{\si^{-1}}}
\newcommand{\gsiiii}{g_{\si^{-1}}^{-1}}
\newcommand{\gtau}{g_{\tau}}

\newcommand{\mv}{\mathcal{V}}

\newcommand{\gmkb}{\mathbb{G}_{m,\kb}}

\newcommand{\slll}{SL_{3,\mathbb{C}}}

\newcommand{\mb}{\mathcal{D}}

\begin{document}
\title{Toric varieties and spherical embeddings over an arbitrary field}
\author{Mathieu Huruguen}

\date{}

\maketitle
\begin{abstract}
We are interested in two classes of varieties with group action, namely toric varieties and spherical embeddings. They are classified by combinatorial objects, called 
fans in the toric setting, and colored fans in the spherical setting. We characterize those combinatorial objects corresponding to varieties defined over an arbitrary field 
$k$. Then we provide some situations where toric varieties over $k$ are classified by Galois-stable fans, and spherical embeddings over $k$ 
by Galois-stable colored fans. Moreover, we construct an example of a smooth toric variety under a $3$-dimensional nonsplit torus over $k$ whose fan is 
Galois-stable but which admits no $k$-form. In the spherical setting, we offer an example of a spherical homogeneous space $X_0$ over $\mr$ of rank $2$ under the action of 
$SU(2,1)$ and a smooth embedding of $X_0$ whose fan is Galois-stable but which admits no $\mr$-form.
\end{abstract}

\section*{Introduction}
In the early $70$'s, Demazure (\cite{De}) gave a full classification of smooth
toric varieties under a split torus in terms of combinatorial objects which he
named fans. This classification was then extended to all toric varieties under
a split torus during the next decade (see for example \cite{Da}).

In the first part of this paper, we address the classification problem for a
nonsplit torus $T$ over a field $k$. Let $K$ be a Galois extension of $k$ which
splits $T$. Then the Galois group $\op{Gal}(K|k)$ acts on fans corresponding to
toric varieties under $T_K=T\times_k K$, and one can speak of Galois-stable
fans. The classification Theorem \ref{criterion} says, roughly speaking, that
toric varieties under a nonsplit torus $T$ are classified by Galois-stable fans
satisfying an additional condition, named $(ii)$. For a quasi-projective fan
(see Proposition \ref{quasi}) condition $(ii)$ holds. If the torus $T$ is of
dimension $2$ then every fan is quasi-projective, and then condition $(ii)$
holds. If $T$ is split by a quadratic extension, condition $(ii)$ is also
automatically satisfied, by a result of Wlodarczyk (\cite{Wl}) which asserts
that any two points in a toric variety (under a split torus) lie on a common
affine open subset. In these situations, toric varieties under $T$ are thus
classified by Galois-stable fans (Theorem \ref{two}). Condition $(ii)$ is
nonetheless necessary, as we construct an example of a three dimensional torus
$T$ over a field $k$, split by an extension $K$ of $k$ of degree $3$, and a
smooth toric variety under the split torus $T_K$ with Galois-stable fan but
which is not defined over $k$ (Theorem \ref{extor}). This provides an example of
a toric variety (under a split torus) containing three points which do not lie
on a common affine open subset.

Recently, Elizondo, Lima-Filho, Sottile and Teitler also studied toric
varieties under nonsplit tori (\cite{El}). They also obtain the statement of
Theorem \ref{two}, but do not address the problem of descent in general. Using
Galois cohomology, they are able to classify toric $k$-forms of
$\mathbb{P}^n_K$ if the extension $K$ of $k$ is cyclic and, toric smooth
surfaces in general.

In the second part of this paper, we address the classification problem for
embeddings of spherical homogeneous spaces. A homogeneous space $(X_0,x_0)$
under a connected reductive group $G$ over $k$ is called spherical if there is
a Borel subgroup $B$ of $\Gkb$ with $Bx_0$ open in $X_0(\kb)$. An embedding of
$X_0$ is a normal $G$-variety over $k$ containing $X_0$ as an open orbit. The
main difference with the toric case is the base point, introduced in order to
kill automorphisms.

The classification of spherical embeddings was obtained by Luna and Vust
(\cite{Lu}) when $k$ is algebraically closed of characteristic $0$, and
extended by Knop (\cite{Kn}) to all characteristics. The classifying objects, called colored fans,
are also of combinatorial nature.

In Section \ref{galoisaction}, we show that the Galois group $\op{Gal}(\kb|k)$
acts on those colored fans, so that we can speak of Galois-stable colored
fans. The main classification theorem is Theorem \ref{criterion2}; like Theorem
\ref{criterion} it asserts that the embeddings of $X_0$ are classified by Galois-stable 
fans satisfying an additional assumption, named $(ii)$. We provide some
situations where this condition $(ii)$ is automatically satisfied (including
the split case, which is not a part of the Luna-Vust theory), and an example of a homogeneous space $X_0$ over $\mr$, under the action of $SU(2,1)$,
and an embedding $X$ of $X_{0,\mc}$ with Galois-stable colored
fan but which is not defined over $\mr$. This gives an example of a smooth
spherical variety containing two points which do not lie on a common affine
subset.

A motivation for studying embeddings of spherical homogeneous spaces is to
construct equivariant smooth compactifications of them. In the toric case, this
construction is due to Colliot-Th\'{e}l\`{e}ne, Harari and Skorobogatov (\cite{Co}).

I'd like to thank M.Brion for his precious advice about that work and his
careful reading. I'd also like to thank E.J.Elizondo, P.Lima-Filho, F.Sottile
and Z.Teitler for communicating their work to me.

\section{Classification of toric varieties over an arbitrary field}\label{toric}
Let $k$ be a field and $\kb$ a fixed algebraic closure. We denote by $T$ a
torus defined over $k$. By a variety over $k$ we mean a separated geometrically
integral scheme of finite type over $k$. We define toric varieties under the
action of $T$ in the following way :

\begin{defi}\label{torvar}
A toric variety over $k$ under the action of $T$ is a normal $T$-variety $X$
such that the group $T(\bar{k})$ has an open orbit in $X(\kb)$ in which it acts
with trivial isotropy subgroup scheme. A morphism between toric varieties under
the action of $T$ is a $T$-equivariant morphism defined over $k$.
\end{defi}
It follows from the definition that a toric variety $X$ under the action of $T$
contains a principal homogeneous space under $T$ as a $T$-stable open subset.
We will denote it by $X_0$.
\begin{defi}
We will say that $X$ is split if $X_0$ is isomorphic to $T$, that is to say, if
$X_0$ has a $k$-point.
\end{defi}
\begin{rem}\label{auto}
If $X$ is a split toric variety, then the automorphism group of $X$ is the
group $T(k)$.
\end{rem}
In the rest of this section, we classify the toric varieties under the action
of $T$ (up to isomorphism). Assuming first that the torus $T$ is split, we
recall how the classification works in terms of combinatorial data named fans
(Section \ref{splitt}). In Section \ref{forms}, we derive the general case from
the split case. We show (Theorem \ref{criterion}) that toric varieties under
the action of $T$ are, roughly speaking, classified by Galois-stable fans
satisfying an additional assumption. In Section \ref{applications}, we provide
some situations where this additional assumption is always satisfied, and an
example where it is not.
\subsection{The split case}\label{splitt}
In this section, we assume that the torus $T$ is split, and give the
classification of toric varieties under the action of $T$. This classification
was obtained by Demazure in the case of smooth toric varieties, and by many
other people in the general case. See \cite{Da} for more details and proofs.
\begin{prop}
Every toric variety under the action of $T$ is split.
\end{prop}
\begin{proof}
By Hilbert's $90$ theorem, every principal homogeneous space under $T$ has a
$k$-point.
\end{proof}
In order to state the main theorem of this section, we need more notations and
definitions.
\begin{notation}
We denote by $M$ the character lattice of $T$, and by $N$ the lattice of one
parameter subgroups, which is dual to $M$. We call $V$ the $\mq$-vector space
$N\otimes_{\mz}\mq$, and $V^*=M\otimes_{\mz}\mq$ its dual.
\end{notation}
By a cone in $V$, we mean the $\mq^+$-linear span of finitely many elements of
$V$. We say that a cone is strictly convex if it contains no line.
\begin{notation}
 If $\cat$ is a cone in $V$, we denote by $\cat^{\vee}\subseteq V^*$ its dual cone, and by $\op{Int}(\cat)$ its relative interior.
\end{notation}
Recall the classical :
\begin{defi}
A fan in $V$ is a finite collection $\me$ of strictly convex cones in $V$
satisfying :
\begin{itemize}
\item Every face of $\cat\in\me$ belongs to $\me$.
\item The intersection of two cones in $\me$ is a face of each.
\end{itemize}
\end{defi}
Let us now recall how a fan can be associated to a toric variety under the
action of $T$. Fix a toric variety $X$ under $T$. If $x\in X(\kb)$, the orbit
of $x$ under $T(\kb)$ is actually defined over $k$. For simplicity, we will
denote by $Tx$ this orbit. If $\omega$ is a $T$-orbit, the open subset 
$$X_{\omega}:=\{x\in X(\kb),\quad \omega\subseteq \overline{Tx}\}$$
is affine and defined over $k$. Fix $x \in X_0(k)$. One can show that the subset
$$\{\lambda\in N,\quad \op{lim}_{t\rw 0} \lambda(t)x \text{ exists in }X \text{ and belongs to} X_{\omega}\}$$
of $N$ is a finitely generated monoid whose $\mq^+$-linear span $\cat_{\omega}$ is a
strictly convex cone in $V$. Observe that the cone $\cat_{\omega}$ does not depend on the point
$x\in X_0(k)$.
 \begin{thm}
By mapping a toric variety $X$ to the collection 
$$\{\cat_{\omega},\quad \omega\subseteq X\text{ is a }T\text{-orbit}\},$$
one gets a bijection between (isomorphism classes of) toric varieties under $T$
and fans in $V$. We will denote by $\me_X$ the fan associated to the toric
variety $X$, and by $X_{\me}$ the toric variety associated to the fan $\me$.
\end{thm}
To construct a toric variety under $T$ out of a fan $\me$ in $V$, one proceeds as
follows. If $\me$ contains only one maximal cone $\cat$, then the variety
$X_{\me}$ is $$X_{\me}=\op{Spec}(k[\cat^{\vee}\cap M]).$$ Observe that
$X_{\me}$ is affine. In the general case, one glues the toric varieties
$(X_{\cat})_{\cat\in \me}$ along their intersections : $X_{\cat}\cap
X_{\cat'}=X_{\cat\cap\cat'}$.

The following proposition enables us to detect the quasi-projectivity of a
toric variety by looking at its fan :
\begin{prop}\label{quasi}
Let $\me$ be a fan in $V$. The variety $X_{\me}$ is quasi-projective if and
only if there exists a family of linear forms $(l_{\cat})_{\cat\in\me}$ on $V$
satisfying the following conditions :
\begin{itemize}
\item $\forall \quad\cat,\cat'\in\me,\quad
l_{\cat}=l_{\cat'}\text{ over } \cat\cap\cat' .$
\item $\forall\quad \cat,\cat'\in\me, \quad \forall x\in \op{Int}(\cat),\quad l_{\cat}(x)>l_{\cat'}(x).$
\end{itemize}
In this situation, we say that the fan $\me$ is quasi-projective.
\end{prop}
\begin{rem}\label{deux}
Every two-dimensional fan is quasi-projective.
\end{rem}
\begin{rem}\label{deuxprime}
If the fan $\me$ has one maximal cone, then it is quasi-projective because
$X_{\me}$ is affine. If $\me$ has two maximal cones, then it is also
quasi-projective. Indeed, let $l\in V^*$ be positive on the first maximal cone
$\cat_1$, and negative on the second $\cat_2$. Putting $l_{\cat_1}=l$ and
$l_{\cat_2}=0$ one gets the result.
\end{rem}

\subsection{Forms of a split toric variety}\label{forms}
In this section, we go back to the general setting ($T$ is not necessarily
split). We fix a finite Galois extension $K$ of $k$ with Galois group $\Gamma$
such that the torus $T_K$ is split. The notations $M,N,...$ will refer to the
corresponding objects associated to $T_K$ in Section \ref{splitt}. These
objects are equipped with an action of the group $\Gamma$.

Fix a toric variety $X$ under $T_K$. We address the following problem :
\begin{question}\label{gen}
Does $X$ admit a $k$-form?
\end{question}
By a $k$-form of $X$, we mean a toric variety $Y$ under the action of $T$, such
that $Y_K\simeq X$ as $T_K$-varieties. Let $F_X$ be the set of isomorphism
classes of $k$-forms of $X$.  We denote by $\open$ the open orbit of $T_K$ in
$X$, and define $F_{\open}$ similarly. By definition, $F_{\open}$ is the set of
principal homogeneous spaces under $T$ which become trivial under $T_K$. By
mapping a $k$-form of $X$ to the principal homogeneous space that it contains,
one obtains a natural map 
$$\delta_X : F_X\rw F_{\open}$$
\begin{question}\label{part}
What can be said about the map $\delta_X$ ?
\end{question}
Theorem \ref{criterion} will give an answer to Questions \ref{gen} and
\ref{part}. We will obtain a criterion involving the fan $\me_X$ for the set
$F_X$ to be nonempty, and show that if this criterion is satisfied, the map
$\delta_X$ is a bijection. As usual in Galois descent issues, semi-linear
actions on $X$ respecting the ambient structure turn out to be very helpful.
\begin{defi}
An action of $\Gamma$ on $X$ is called toric semi-linear if for every
$\si\in\Gamma$ the diagrams 
$$\begin{CD}
X @>\si>>  X \\
@VVV  @VVV \\
\op{Spec}(K) @>\si>>  \op{Spec}(K)
\end{CD},\quad \begin{CD}
T_K\times X @>>>  X \\
@VV(\si,\si)V      @VV\si V \\
T_K\times X @>>>  X
\end{CD}$$
are commutative. The group $\akx=T(K)$ acts by conjugacy on the whole set of toric semi-linear
actions of $\Gamma$ on $X$. We denote by $E_X$ the set of conjugacy classes of
toric semi-linear actions of $\Gamma$ on $X$.
\end{defi}
If $Y$ is a $k$-form of $X$ and if one fixes a $T_K$-equivariant isomorphism
$X\rw Y_K$, then one can let the group $\Gamma$ act on $X$. Replacing the
isomorphism $X\rw Y_K$ by another one, one obtains a $T(K)$-conjugated toric
semi-linear action. This proves that there is a natural map 
$$\alpha_X : F_X\rw E_X$$
The following proposition is part of the folklore and can be found for example
in \cite{El}.
\begin{prop}\label{classical}
The map $\alpha_X$ is injective. A toric semi-linear action on $X$ is in the
image of $\alpha_X$ if and only if the quotient $X/\Gamma$ exists, or, in other
words, if and only if one can cover $X$ by $\Gamma$-stable quasi-projective
subsets.
\end{prop}
\begin{rem}\label{ez}
If the variety $X$ itself is quasi-projective, the map $\alpha_X$ is thus
bijective.
\end{rem}
\begin{prop}
The map $\alpha_{\open}$ is bijective. Otherwise stated, a principal
homogeneous space under $T$ which becomes trivial under $T_K$ is characterized
by the toric semi-linear action it induces on $\open$.
\end{prop}
\begin{proof}
Use Proposition \ref{classical} and observe that $\open$ is affine.
\end{proof}
\begin{rem}
The set $F_{\open}$ is naturally the Galois cohomology set $H^1(\Gamma,T(K))$.
\end{rem}
The following proposition (also obtained in \cite{El}) enables us to see very
easily whether the set $E_X$ is empty or not, by looking at the fan $\me_X$.
\begin{prop}\label{res}
The set $E_X$ is nonempty if and only if the fan $\me_X$ is $\Gamma$-stable in
the sense that, for every cone $\cat\in\me_X$, and for every $\si\in\Gamma$,
the cone $\si(\cat)$ still belongs to $\me_X$. In this case, the restriction
map $E_X\rw E_{\open}$ is a bijection.
\end{prop}
\begin{rem}
 The open orbit $\open$ is easily seen to be $\Gamma$-stable for any toric semi-linear action
of $\Gamma$ on $X$. By mapping such an action to its restriction to $\open$ one
gets what we call the restriction map $E_X\rw E_{\open}$.
\end{rem}

\begin{proof}
Assume first that the set $E_X$ is non empty. The variety $X$ is thus endowed
with a toric semi-linear action of $\Gamma$. Let $\omega$ be an orbit of $T_K$
on $X$, and $\si$ be an element of $\Gamma$. Let $x$ be a $K$-point in $\open$
(it exists because $T_K$ is split). One has 
$$\si(\cat_{\omega})\cap N=\{\si(\lambda),\quad \lambda\in N,\quad \op{lim}_{t\rw 0} \lambda(t)x \text{ exists in }X \text{ and belongs to
} X_{\omega}\}.$$ Thus,  
$$\si(\cat_{\omega})\cap N=\{\lambda\in N,\quad \op{lim}_{t\rw 0} \lambda(t)\si(x) \text{ exists in }X \text{ and belongs to
} X_{\si(\omega)}\}.$$ But the $K$-point $\si(x)$ belongs to $\open$, showing
that $\si(\cat_{\omega})=\cat_{\si(\omega)}.$ This proves that $\me_X$ is
stable under the action of $\Gamma$.

Assume now that $\me_X$ is $\Gamma$-stable. Let $\cat$ be a cone in $\me$ and
$\si$ an element of $\Gamma$. The linear map $\si$ on $M$ gives a morphism of
monoids 
$$\si(\cat)^{\vee}\cap M\rw \cat^{\vee}\cap M $$
and then a semi-linear morphism of $K$-algebras 
$$K[\si(\cat)^{\vee}\cap M]\rw K[\cat^{\vee}\cap M],$$
inducing a morphism of varieties 
$$U_{\cat}\rw U_{\si(\cat)}$$
which respects the toric structures on both sides. These morphisms patch
together, and enable us to construct the desired toric semi-linear action on
$X$. This completes the proof of the first point.

Suppose from now on that the set $E_X$ is non empty, and fix a toric
semi-linear action of $\Gamma$ on $X$. Denote by $*$ a toric semi-linear action
of $\Gamma$ on $\open$. Then, for all $\si\in \Gamma$, the morphism 
$$\open\rw\open$$
$$x\mapsto \si^{-1}*(\si(x))$$
is a toric automorphism of $\open$, that is, the multiplication by an element
of $T(K)$. But such a multiplication extends to $X$, proving that the
semi-linear action $*$ of $\Gamma$ extends to $X$. In other words, the
restriction map is surjective. But it is also injective because $\open$ is open
in $X$. This completes the proof of the proposition.
\end{proof}
\begin{rem}\label{rem}
By the previous arguments, one sees that if $\me$ is $\Gamma$-stable, if
$\omega$ is an orbit of $T_K$ in $X$, and $\si$ an element of $\Gamma$, the
notation $\si(\omega)$ makes sense, and does not depend of the chosen toric
semi-linear action of $\Gamma$ on $X$. Moreover, one sees that for every toric
semi-linear action of $\Gamma$ on $X$, and for every cone $\cat\in\me$,
$\si(X_{\cat})=X_{\si(\cat)}$.
\end{rem}
We are now able to answer Questions \ref{gen} and \ref{part} :
\begin{thm}\label{criterion}
The set $F_X$ is non empty if and only if the two following conditions are satisfied :
\begin{enumerate} [(i)]
\item The fan $\me_X$ is $\Gamma$-stable.
\item For every cone $\cat\in\me_X$, the fan
consisting of the cones $(\si(\cat))_{\si\in\Gamma}$ and their faces is
quasi-projective.
\end{enumerate}
In that case, the map $\delta_X$ is bijective. Otherwise stated, for every
principal homogeneous space $X_0$ under $T$, there is a unique $k$-form of $X$
containing $X_0$, up to isomorphism.
\end{thm}
\begin{proof}
Assume first that conditions $(i)$ and $(ii)$ are fulfilled. By Proposition
\ref{res}, the set $E_X$ is nonempty. Fix a toric semi-linear action of
$\Gamma$ on $X$. By condition $(ii)$ and Proposition \ref{quasi}, for every
cone $\cat\in\me$, the open subset $\bigcup_{\si\in\Gamma}X_{\si(\cat)}$ is
quasi-projective. But these open subsets are $\Gamma$-stable (by Remark
\ref{rem}) and cover $X$. By Proposition \ref{classical} the quotient
$X/\Gamma$ exists. Performing this argument for every toric semi-linear action
of $\Gamma$ on $X$, one proves that the map $\delta_X$ is bijective. In
particular, the set $F_X$ is nonempty.

Assume now that the set $F_X$ is nonempty. We want to prove that conditions
$(i)$ and $(ii)$ are fulfilled. By Proposition \ref{classical}, the set $E_X$
is nonempty, so that condition $(i)$ holds. Fix a $T$-orbit $\omega$ on $X$. By
Remark \ref{rem}, the open subset $U=\bigcup_{\si\in\Gamma}X_{\si(\omega)}$ is
$\Gamma$-stable. Moreover, closed $T$-orbits in $U$ form a unique orbit under
$\Gamma$. There exists therefore an affine open subset $V$ in $U$ intersecting
every closed $T$-orbit. By Proposition \ref{Su}  below, one concludes that $U$
is quasi-projective, or, using Proposition \ref{quasi}, that the fan consisting
of the cones $(\si(\cat_{\omega}))_{\si\in\Gamma}$ and their faces is
quasi-projective. This being true for every $T$-orbit $\omega$, we are done.
\end{proof}
Let $G$ be a linear algebraic group over $k$. Sumihiro proved in \cite{Sum} that a normal $G$-variety containing only one closed orbit is quasi-projective.
By the same arguments one gets the next proposition, whose proof we give for the convenience of the reader.
\begin{prop} \label{Su}
Let $X$ be a normal $G$-variety over $k$. Assume that there is an affine subset
of $X$ which meets every closed orbit of $G$ on $X$. Then $X$ is
quasi-projective.
\end{prop}
\begin{proof}
Quasi-projectivity is of geometric nature, so one can assume that $k$ is
algebraically closed in what follows.  We will use the following
quasi-projectivity criterion given in \cite{Sum} (Lemma $7.$) :
\begin{lemma}
If there exists a line bundle $\ml$ on $X$ and global sections $s_1,...,s_n$
generating $\ml$ at every point and such that $X_{s_1},...,X_{s_n}$ are affine,
then $X$ is quasi-projective.
\end{lemma}
Let $U$ be an affine subset of $X$ meeting every closed orbit of $G$ and
$D=X\setminus U$. Then $D$ is a Weil divisor. Denote by $\ml$ the coherent
sheaf $\mo_X(D)$, and by $i:X^0\rw X$ the inclusion of the regular locus of $X$
in $X$. The natural morphism $\varphi : \ml\rw i_*(\ml_{|X^0})$ is an
isomorphism, because $X$ is normal. Moreover, the sheaf $\ml_{|X^0}$ is
invertible on the smooth $G$-variety $X^0$, so that one can let a finite
covering $G'$ of $G$ act on $\ml_{|X^0}$, and thus on $\ml$, using the
isomorphism $\varphi$. This enables us to see that the set 
$$A=X\setminus\{x\in X,\quad \ml\text{ is invertible in a neighborhood of }x\}$$
is a $G$-stable closed subset of $X$. Moreover, $\ml_{|U}$ is trivial, so that
$A$ is contained in $D$, proving that $A$ is empty, or, otherwise stated, the
sheaf $\ml$ is invertible : the divisor $D$ is Cartier. Let now $s$ be the
canonical global section of the sheaf $\ml$. The common zero locus of the
translates $(g'.s)_{g'\in G'}$ is also a $G$-closed subset of $D$, and is
therefore empty.  We can thus find a finite number of elements $g'_1,...,g'_n$
of $G'$ such that the global sections $g'_1.s,...,g'_n.s$ generate $\ml$ at
every point. Moreover, the open subsets $X_{g'_i.s}$ are $G$-translates of $U$
and are therefore affine. This completes the proof.

\end{proof}

\subsection{Applications}\label{applications}
 In this section we
give first some conditions on $T$ for every toric variety under $T_K$ to have a
$k$-form. Then we construct an example of a $3$-dimensional torus $T$ split by
a degree $3$ extension $K$ of $k$ and a toric variety $X$ under $T$ with $F_X$
empty.
\begin{thm}\label{two}
Assume that $\op{dim}_k(T)=2$, or that the torus $T$ is split by a quadratic
extension. Then, for every toric variety $X$ under the action of $T_K$ whose
associated fan is $\Gamma$-stable, the map $\delta_X$ is bijective.
\end{thm}
\begin{proof}
If $\op{dim}_k(T)=2$, then $\op{dim}_K T_K=2$, and Remark \ref{deux} shows that
the variety $X$ is quasi-projective. By Remark \ref{ez}, the map $\delta_X$ is
bijective in that case. Assume now that the torus $T$ is split by a quadratic
extension $K$ of $k$. Then the Galois group $\Gamma$ has two elements, and thus
for every cone $\cat\in\Sigma$, the fan consisting of the cones
$(\si(\cat))_{\si\in\Gamma}$ and their faces has only one or two maximal cones.
By Remark \ref{deuxprime}, this fan is automatically quasi-projective, so that
Theorem \ref{criterion} applies.
\end{proof}
\begin{rem}
The tori $T$ split by a fixed quadratic extension $K$ of $k$ are exactly the
products of the following three tori : $\mathbb{G}_{m}$,
$R_{K|k}(\mathbb{G}_{m})$, $R^1_{K|k}(\mathbb{G}_{m})$. Here
$R_{K|k}(\mathbb{G}_{m})$ is the Weil restriction of the torus
$\mathbb{G}_{m}$, and $R^1_{K|k}(\mathbb{G}_{m})$ is the kernel of the norm map
$$N_{K|k} : R_{K|k}(\mathbb{G}_{m})\rw \mathbb{G}_{m}.$$
\end{rem}
We now provide an example of a split toric variety over $K$ which admits no
$k$-form. We assume that the Galois extension $K$ over $k$ is cyclic of degree
$3$. Let $\ep$ be a generator of $\Gamma$. Fix a basis $(u,v,w)$ of the lattice
$N=\mz^3$, and let $\ep$ acts on $N$ by the following matrix in the basis
$(u,v,w)$
$$\left(
\begin{array}{ccc}
0 & -1 & 0  \\
1 & -1 & 0  \\
0 &  0 & 1  \\
\end{array}
\right).$$ One defines in this way an action of $\Gamma$ on $N$. Let $T$ be the
torus over $k$ corresponding to this action. This is a three-dimensional torus
split by $K$.

\begin{lemma}
Let $\cat=\operatorname{Cone}(5u+v-5w,-5u-5v+14w,4u-v)$. Then
$\cat\cap\ep(\cat)=\{0\}.$
\end{lemma}
\begin{proof}
One has $\ep(\cat)=\operatorname{Cone}(-u+4v-5w,5u+14w,u+5v)$. Let $x\in
\cat\cap\ep(\cat)$. There exist $a,b,c,d,e,f\geqslant 0$ such that
$$
\begin{aligned}
x &= (5a-5b+4c)u+(a-5b-c)v+(-5a+14b)w\\
 &=(-d+5e+f)u+(4d+5f)v+(-5d+14e)w.
\end{aligned}
$$
Consequently 
$$\left\{\begin{array}{ccc}
 5a-5b+4c & = & -d+5e+f \\
a-5b-c & = & 4d+5f \\
-5a+14b & = &-5d+14e,
\end{array}\right.$$
 and then $14$(second line)+$5$(third line) gives
$$\begin{aligned}
 -11a-14c=31d+70e+70f.
\end{aligned}$$
The left hand-side is nonpositive, and the right hand-side is nonnegative, so that
$d=e=f=0$, and then $x=0$.

\end{proof}
\begin{defi}
Let $\me$ be the fan consisting of  $\cat, \ep(\cat), \ep^2(\cat)$ and their
faces.
\end{defi}
\begin{lemma}\label{lisse}
The fan $\me$ is smooth.
\end{lemma}
\begin{proof}
It is enough to check that $\cat$ is a smooth cone, and this holds :
$$\left|
\begin{array}{ccc}
5  & -5 &  4  \\
1  & -5 & -1  \\
-5 & 14 &  0  \\
\end{array}
\right|=1.$$
\end{proof}

\begin{prop} \label{strictconv}
The fan $\me$ is not quasi-projective.
\end{prop}
\begin{proof}
Observe first that $-w\in \cat-\ep(\cat)$. Indeed,
$45u-25w=5((5u+v-5w)+(4u-v))\in \cat$ and $45u+126w=9(5u+14w)\in\ep(\cat)$, so
that $-151w=(45u-25w)-(45u+126w)\in \cat-\ep(\cat)$, and finally, $-w\in
\cat-\ep(\cat)$. Applying $\ep$ and $\ep^2$, one obtains 
$$-w\in (\cat-\ep(\cat))\cap(\ep(\cat)-\ep^2(\cat))\cap (\ep^2(\cat)-\cat).$$ Now suppose that the fan $\me$ is
quasi-projective, and use the notations of Definition \ref{quasi}. The linear
form $l_{\cat}-l_{\ep(\cat)}$ is strictly positive on
$(\cat-\ep(\cat))\setminus\{0\}$, and then : $(l_{\cat}-l_{\ep(\cat)})(-w)>0$.
But similarly, $l_{\ep(\cat)}-l_{\ep^2(\cat)}$ is strictly positive on
$(\ep(\cat)-\ep^2(\cat))\setminus\{0\}$ (resp.
$(\ep^2(\cat)-\cat)\setminus\{0\})$ and thus
$(l_{\ep(\cat)}-l_{\ep^2(\cat)})(-w)>0$ (resp.
$(l_{\ep^2(\cat)}-l_{\cat})(-w)>0$). This gives a contradiction, proving that
the fan $\me$ is not quasi-projective.
\end{proof}
In view of Theorem \ref{criterion}, one has the following :
\begin{thm}\label{extor}
The toric variety $X_{\me}$ under $T_K$ is smooth and does not admit any
$k$-form.
\end{thm}
In this example, $T$ is of minimal dimension and $\Gamma$ of minimal
order, in view of Theorem \ref{two}.
\begin{rem}
The variety $X_{\me}$ gives an example of a toric variety containing three
points which do not lie on an open affine subset. In \cite{Wl}, it is shown
that any two points in a toric variety lie on a common affine subset.
\end{rem}
\begin{rem}
Using techniques from \cite{Co}, one can "compactify" the fan $\me$ in a
$\Gamma$-equivariant way, and thus produce an example of a smooth complete
toric variety under the action of $T_K$ which has no $k$-form. We first produce
a complete simplicial fan $\Gamma$-stable fan $\me_0$ containing $\me$. For
simplicity, we note 
$$r_1=-5u-5v+14w, \quad s_1=4u-v, \quad t_1=5u+v-5w,$$
$$r_2=\ep(r_1),\quad s_2=\ep(s_1),\quad t_2=\ep(t_1), \quad r_3=\ep^2(r_1),\quad s_3=\ep^2(s_1),\quad t_3=\ep^2(t_1).$$
The fan $\me_0$ has maximal cones 
$$\op{Cone}(r_1,t_3,s_1),\quad\op{Cone}(t_3,s_1,t_1), \quad\cat=\op{Cone}(r_1,s_1,t_1)$$
$$\op{Cone}(r_1,r_2,t_1),\quad\op{Cone}(r_1,r_2,r_3),\quad\op{Cone}(t_1,t_2,t_3)$$
and their images under $\Gamma$. Because the fan $\me_0$ is complete, it gives
a triangulation of the unit sphere in $\mr^3$. By projecting this triangulation
from the South Pole to the tangent plane of the North Pole, one gets Figure \ref{fann}. The maximal cone $\op{Cone}(t_1,t_2,t_3)$ of $\me_0$ is missing in this picture
because it is sent to infinity by the projection. In order to smoothen the fan
$\me_0$ in a $\Gamma$-equivariant way, we use the method described in
\cite{Co}. We first subdivide $\me_0$ using the vectors $r=-5v+28w$, $w$,
$t=u-4v-10w$ and $-w$. We thus obtain a complete, simplicial $\Gamma$-stable
fan $\me'_0$ satisfying Property $(*)$ defined in Proposition $2.$ of
\cite{Co}. This fan $\me'_0$ has maximal cones

$$\op{Cone}(r_1,r,t_1),\quad \op{Cone}(r,r_2,t_1),\quad\op{Cone}(r_1,r,w),\quad\op{Cone}(r,r_2,w),\quad \cat=\op{Cone}(r_1,s_1,t_1)$$
$$\op{Cone}(r_1,s_1,t_3)\quad \op{Cone}(t_3,s_1,t)\quad\op{Cone}(t,t_1,s_1)\quad\op{Cone}(t_3,t,-w)\quad\op{Cone}(t,t_1,-w)$$
and their images under $\Gamma$. We now can apply to $\me'_0$ the algorithm
explained in the proof of Proposition $3.$ of \cite{Co} because it satisfies
Property $(*)$. We thus get a complete, smooth $\Gamma$-stable fan containing
$\me$.
\begin{figure}[!ht]\label{fann}
\centering
\includegraphics[width=8cm,height=8cm]{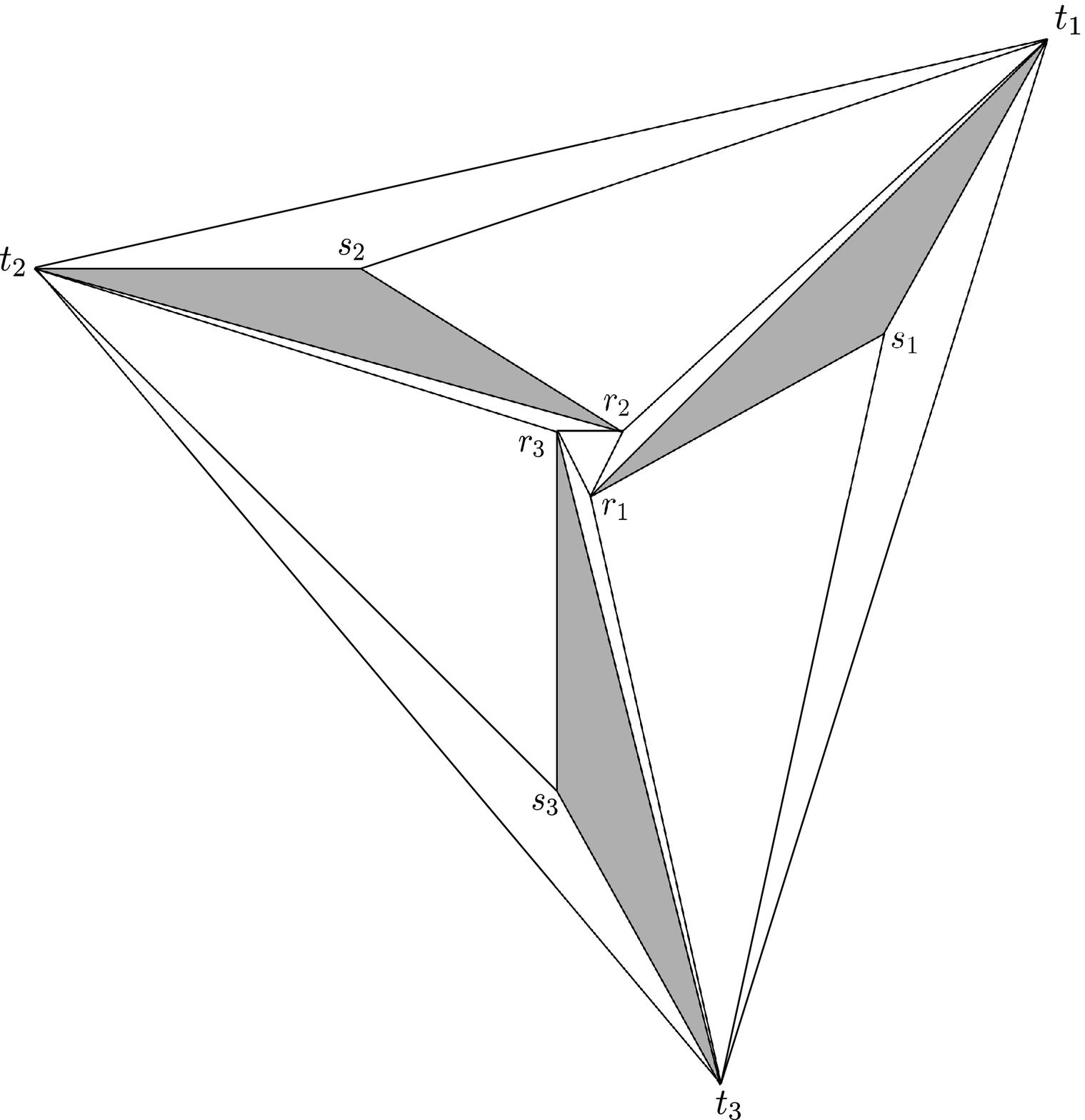}
\caption{The fans $\me$ (grey) and $\me_0$.}
\end{figure}
\end{rem}

\section{Spherical embeddings over an arbitrary field}\label{spherical}
Let $k$ be a field, and $\kb$ a fixed algebraic closure. Throughout this
section, we denote by $\Gamma$ the absolute Galois group of $k$. Let $G$ be a
connected reductive algebraic group over $k$. We note $\Gkb=G\times_k\kb$.
\begin{defi}\label{defsf}
 A spherical homogeneous space under $G$ over $k$ is a pointed $G$-variety $(X_0,x_0)$ over $k$ such that :
\begin{itemize}
\item The group $G(\kb)$ acts transitively on $X_0(\kb)$.
\item $x_0\in X_0(k)$.
\item There exists a Borel subgroup $B$ of $\Gkb$ such that the orbit of $x_0$ under $B(\kb)$ is open in $X_0(\kb)$.
\end{itemize}
We say that $X_0$ is split over $k$ if there exists a split Borel subgroup $B$
of $G$ such that $B x_0$ is open in $\Xokb$. In this case, the group $G$ itself
is split.
\end{defi}
\begin{rem} For simplicity, we will often denote by $X_0$ a spherical homogeneous
space, omitting the base point.
\end{rem}
\begin{defi}
 Let $X_0$ be a spherical homogeneous space under $G$. An embedding of $X_0$ is a pointed normal $G$-variety $(X,x)$ together with a $G$-equivariant immersion $i: X_0\rw X$
preserving base points. A morphism between two embeddings is defined to be a
$G$-equivariant morphism defined over $k$ preserving base points.
\end{defi}
\begin{rem}
For simplicity again, when speaking about an embedding of $X_0$, we will often
omit the base point and the immersion.
\end{rem}
\begin{rem}
If there exists a morphism between two embeddings, then it is unique.
\end{rem}

In this section, we classify the embeddings of a fixed spherical homogeneous
space $X_0$ under $G$ up to isomorphism. We first recall in Section
\ref{warmup} the Luna-Vust classification theory of spherical embeddings,
assuming that the field $k$ is algebraically closed. Fundamental objects named
colored fans are the cornerstone of that theory. In Section \ref{galoisaction}
we let the Galois group $\Gamma$ act on these colored fans, and prove in
Section \ref{kformspheric} that the embeddings of $X_0$ are classified by
$\Gamma$-stable colored fans satisfying an additional condition. In Section
\ref{appli} we provide several situations where this condition is fulfilled
(including the split case), and an example where it is not.

\subsection{Recollections on spherical embeddings}\label{warmup}
We assume that $k$ is algebraically closed. Fix a spherical homogeneous space
$X_0$ under $G$, and a Borel subgroup $B$ such that $B x_0$ is open in $X_0$.
We now list some facts about $X_0$ and give the full classification of its
embeddings (see \cite{Kn} for proofs and more details). We will introduce along
the way notations that will be systematically used later on.
\begin{notation}
We will denote by :
\begin{itemize}
\item $\ka=k(X_0)$ the function field of $X_0$. The group $G$ acts on $\ka$.
\item $\nkb$ the set of $k$-valuations on $\ka$ with values in $\mq$. The group $G$ acts on $\nkb$. We will denote by $\nkb^G$ the set of $G$-invariant
valuations, and by $\nkb^B$ the set of $B$-invariant valuations.
\item $\open$ the orbit of $x_0$ under the action of $B$. This is an affine variety.
\item $\md$ the set of prime divisors in $X_0\setminus\open$. These are finitely many $B$-stable divisors.
\item $\mx$ the set of weights of $B$-eigenfunctions in $\ka$. This is a sublattice of the character lattice of $B$. The rank of $\mx$ is
called the rank of $X_0$ and denoted by $rk(X_0)$.
\item $V=\operatorname{Hom}_{\mathbb{Z}}(\mx,\mathbb{Q})$. This is a $\mq$-vector space of dimension $rk(X_0)$.
\item $\rho$ the  map  $\mv\rightarrow V,\quad\nu\mapsto (\chi\mapsto \nu(f_{\chi})),$
$f_{\chi}\in\ka$ being a $B$-eigenfunction of weight $\chi$ (such an $f_{\chi}$ is
uniquely determined up to a scalar). The restriction of $\rho$ to $\mv^G$ is
injective, and the image $\rho(\mv^G)$ is a finitely generated cone in $V$
whose interior is nonempty.
\end{itemize}
\end{notation}
The Luna-Vust theory of spherical embeddings gives a full classification of
embeddings of $X_0$ in terms of combinatorial data living in the ones we have
just defined.
\begin{defi}
A $\textbf{colored cone}$ inside $V$ with colors in $\md$ is a couple
$(\cat,\mf)$ with $\mf\subseteq \md$, and satisfying :
\begin{itemize}
\item $\cat$ is a cone generated by $\rho(\mf)$ and finitely many elements in
$\rho(\mv^G)$.
\item The relative interior of $\cat$ in $V$ meets $\rho(\mv^G)$.
\item The cone $\cat$ is strictly convex and $0\notin \rho(\mf)$.
\end{itemize}
A colored cone $(\cat',\mf')$ is called a $\textbf{face}$ of $(\cat,\mf)$ if
$\cat'$ is a face of $\cat$ and $\mf'=\mf\cap\rho^{-1}(\cat')$. A $\textbf{
colored fan }$ is a finite set $\me$ of colored cones satisfying :
\begin{itemize}
\item $(0,\emptyset)\in \me$.
\item Every face of $(\cat,\mf)\in\me$ belongs to $\me$.
\item There exists at most one colored cone
$(\cat,\mf)\in\me$ containing a given $\nu\in\rho(\mv^G)$ in its relative
interior.
\end{itemize}
\end{defi}
Let us now recall how a colored fan can be associated to an embedding $X$ of
$X_0$. The open subset $\open$ being affine, its complement $X\setminus \open$
is pure of codimension one, and thus a union of $B$-stable prime divisors. Let
$\omega$ be a $G$-orbit on $X$. We denote by $\mb_{\omega}$ the set of prime
divisors  $D\subseteq X\setminus \open$ with $\omega\subseteq D$, and by
$\mf_{\omega}$ the subset of $D\in\md$ with $\omega\subseteq \overline{D}$. The
elements of $\mf_{\omega}$ are called the $\textbf{colors}$ associated to the
orbit $\omega$. Let $\cat_{\omega}$ be the cone in $V$ generated by the
elements 
$$\rho(\nu_D),\quad D\in\mb_{\omega}$$
where $\nu_D$ is the normalized valuation of $\ka$ associated to the divisor
$D$. We are now able to state the main theorem of this section :
\begin{thm}\label{lunavust}
By mapping $X$ to 
$$\{(\cat_{\omega},\mf_{\omega}),\quad \omega\subseteq X\text{ is a $G$-orbit}\}$$
one gets a bijection between embeddings of $X_0$ and colored fans in $V$, with
colors in the set $\md$. We will denote by $\me_X$ the colored fan associated
to the embedding $X$ of $X_0$, and by $X_{\me}$ the embedding of $X_0$
associated to the colored fan $\me$.
\end{thm}
\begin{rem}
We will say that the embedding $X$ has no colors if, for every orbit $\omega$
of $G$ on $X$, the set $\mf_{\omega}$ is empty.
\end{rem}
We are not going to explain how an embedding of $X_0$ can be construct out of a
colored fan. The curious reader can find the recipe in \cite{Kn}.

In the rest of the section, we list some properties of spherical embeddings
that will be used later. Firstly, as in the toric case, one is able to say
whether an embedding of $X_0$ is quasi-projective by using the associated
colored fan (see \cite{Br2}).
\begin{prop}\label{kazi}
Let $\me$ be a colored fan. The embedding $X_{\me}$ is quasi-projective if and
only if there exists a collection $(l_{\cat,\mf})_{(\cat,\mf)\in\me}$ of linear
forms on $V$ satisfying :
\begin{itemize}
\item $\forall (\cat,\mf)\in\me,\quad\forall (\cat',\mf')\in\me,\quad
l_{\cat,\mf}=l_{\cat',\mf'}$ over $\cat\cap\cat'$.
\item $ \forall (\cat,\mf)\in\me,\quad \forall x\in \op{Int}(\cat)\cap \rho(\mv^G),\quad \forall (\cat',\mf')\in \me\setminus\{(\cat,\mf)\},\quad
l_{\cat,\mf}(x)>l_{\cat',\mf'}(x).$
\end{itemize}
In this situation, we say that the fan $\me$ is quasi-projective.
\end{prop}
We now give some results of local nature on spherical embeddings. Knowing the
local structure of toric varieties (as explained in Section \ref{splitt}) was a
crucial point in order to prove Theorem \ref{criterion}. A toric variety $X$
under $T$ is covered by open affine $T$-stable subset. This fact is not true on
a spherical embedding $X$ of $X_0$, but we still have a nice atlas of affine
charts at our disposal. If $\omega$ is a $G$-orbit on $X$, we denote by 
$$X_{\omega,G}:=\{y\in X, \quad \omega\subseteq \overline{G y}\}.$$
This is a $G$-stable open subset of $X$ containing $\omega$ as its unique
closed orbit. It is quasi-projective by a result of Sumihiro \cite{Sum}. We define 
$$X_{\omega,B}:=X_{\omega,G}\setminus \bigcup_{D\in \md\setminus \mf_{\omega}} \overline{D}.$$
This is a $B$-stable affine open subset of $X$. Moreover, the $G$-translates of
$X_{\omega,B}$ cover $X_{\omega,G}$. One has
$$k[X_{\omega,B}]=\{f\in k[\open],\quad\forall D\in \mb_{\omega} \quad\nu_D(f)\geqslant 0\}.$$
Let 
 $$P:=\bigcap_{D\in\md\setminus\mf_{\omega}}\op{Stab}(D).$$
 This is a
parabolic subgroup of $G$ containing $B$, and the open subset $X_{\omega,B}$ is
$P$-stable. The following theorem describes the action of $P$ in
$X_{\omega,B}$. See \cite{Br4} for a proof.
\begin{thm}\label{strucloc}
There exists a Levi subgroup $L$ of $P$ and a closed $L$-stable subvariety $S$
of $X_{\omega,B}$ containing $x_0$, such that the natural
map 
$$R_u(P)\times S\rw X_{\omega,B},\quad (g,x)\mapsto gx $$
is a $P$-equivariant isomorphism.
\end{thm}
Finally, let us introduce the class of horospherical homogeneous spaces :
\begin{defi}
A homogeneous space $X_0$ is said to be horospherical if the isotropy group of
$x_0$ contains a maximal unipotent subgroup of $G$.
\end{defi}
Any horospherical homogeneous space is spherical. Moreover, thanks to the
following proposition, we are able to recognize horospherical homogeneous
spaces among spherical ones in a very simple way (see \cite{Kn}).
\begin{prop}\label{horosph}
A spherical homogeneous space $X_0$ is horospherical if and only if
$\rho(\mv^G)=V$.
\end{prop}

\subsection{Galois actions}\label{galoisaction}
We return to an arbitrary field $k$ with algebraic closure $\kb$. Recall that
$\Gamma$ is the absolute Galois group of $k$, and $G$ is a connected reductive
algebraic group over $k$. Fix a spherical homogeneous space $X_0$ under the
action of $G$, and a Borel subgroup $B$ of $\Gkb$ such that $B x_0$ is open in
$X_0(\kb)$. In the previous section, we introduced some data attached to
$\Xokb$. In this section, we let the group $\Gamma$ act on these data.
\subsubsection{Action on $\ka$}
The group $\Gamma$ acts on $\ka$ by the following formula 
$$\forall \si\in\Gamma,\quad \forall f\in \ka,\quad\forall x\in X_0(\kb),\quad \si(f)(x)=\si(f(\si^{-1}(x))).$$
\begin{prop}\label{equiv2}
One has
$$\forall\si\in\g,\quad\forall g\in G(\kb),\quad\forall f\in\kbXo,\quad
\si(g f)=\si(g) \si(f).$$
\end{prop}
\begin{proof}
Straightforward computation.
\end{proof}
\subsubsection{Action on $\mv^B$}
If $\si\in \Gamma$, then the group $\si(B)$ is a Borel subgroup of $G$, so that
(see \cite{Bo} Chap.IV, Theorem $11.1$) there exists $g_{\si}\in G(\kb)$
satisfying
$$\si(B)=\gsi B \gsii.$$
Moreover, $\gsi$ is unique up to right multiplication by an element of $B$. We
let the group $\Gamma$ act on $\mv$ by 
$$\forall \si\in \Gamma,\quad \forall \nu\in\mv,\quad \forall f\in\ka,\quad \si(\nu)(f)=\nu(\si^{-1}(\gsi f)).$$

\begin{prop}\label{action}
We define in this way an action of $\Gamma$ on $\mv^B$ which does not depend on
the particular choice of the $(\gsi)_{\si\in\Gamma}$.
\end{prop}
\begin{proof}
Fix $\nu\in\mv^B$, $\si\in \Gamma$ and $f\in\ka$. Let $b\in B$, and $b'\in B$
such that : $\gsi b^{-1}=\si(b')\gsi$. Then one has 
$$(b\si(\nu))(f)=\nu(\si^{-1}(\gsi b^{-1}f))=\nu(b'\si^{-1}(\gsi f))=\si(\nu)(f)$$
because $\nu\in\mv^B$. Thus we have proved that for all $\si\in\Gamma$ and
$\nu\in\mv^B$, $\si(\nu)\in\mv^B$. Now let us check that one defines an action
of $\Gamma$ on $\mv^B$. For this we will need the following lemma :
\begin{lemma}
$$\forall (\si,\tau)\in\g^2, \quad \exists  b_{\si,\tau}\in B \text{ with }\si(\gtau)\gsi=g_{\si\tau}b_{\si,\tau}.$$
\end{lemma}
\begin{proof}
Express $\si(\tau(B))$ in two different ways.
\end{proof}
Fix $(\si,\tau)\in\g^2$, $\nu\in\mv^B$ and $f\in\ka$. Then 
$$(\si\tau)(\nu)(f)=\nu(\tau^{-1}(\si^{-1}(g_{\si\tau}f)))=\nu(\tau^{-1}(\si^{-1}(\si(g_{\tau})\gsi b_{\si,\tau}^{-1})f)))=\nu(\tau^{-1}(g_{\tau}\si^{-1}(\gsi b_{\si,\tau}^{-1}f))).$$
This shows that 
$$(\si\tau)(\nu)=b_{\si,\tau}(\si(\tau(\nu)))=\si(\tau(\nu))$$
because $\si(\tau(\nu))\in \mv^B$. The fact that this action does not depend on
the choice of the $\gsi$ readily follows from the fact that we are working with
$B$-invariant valuations.

\end{proof}

\subsubsection{Action on $\mx$}
Let us denote by $\mx(B)$ the character lattice of the group $B$. For
$\si\in\g$ and $\chi\in\mx(B)$ we define 
$$\si(\chi):B\rightarrow \mathbb{G}_{m,\kb}, \quad b\mapsto \si(\chi(\gsiiii \si^{-1}(b)\gsiii)).$$
As in the proof of Proposition \ref{action}, one checks that this defines an
action of $\Gamma$ on $\mx(B)$ which does not depend on the choice of the
$(\gsi)_{\si\in\Gamma}.$ The following proposition shows that this action
restricts to an action on $\mx$ :
\begin{prop}\label{calcul}
Let $\chi\in\mx$ and $f_{\chi}\in\ka$ a $B$-eigenfunction of weight $\chi$. Fix
$\si\in\g$. Then $\si(\gsiii f)$ is a $B$-eigenfunction of weight $\si(\chi)$.
\end{prop}
\begin{proof}
Straightforward computation.
\end{proof}

\subsubsection{Action on $V$}
Being dual to $\mx$, the vector space
$V=\operatorname{Hom}_{\mathbb{Z}}(\mx,\mathbb{Q})$ inherits a linear action of
$\Gamma$ defined by 

$$\forall \varphi\in V, \quad\forall \si\in\g,\quad \forall\chi\in \mx,\quad
\si(\varphi)(\chi)=\varphi(\si^{-1}(\chi)).$$

\begin{prop}\label{equiva}
The map 
$$\rho :\mv^B\rightarrow V$$
is $\Gamma$-equivariant.
\end{prop}
\begin{proof}
Let $\nu\in\mv^B$, $\si\in\g$ and $\chi\in\mx$. Denote by $f_{\chi}\in\ka$ a
$B$-eigenfunction of weight $\chi$. By Proposition \ref{calcul}, $\si^{-1}(\gsi
f)$ is a $B$-eigenfunction of weight $\si^{-1}(\chi)$. Then 
$$\rho(\si(\nu))(\chi)=\si(\nu)(f_{\chi})=\nu(\si^{-1}(\gsi f))=\rho(\nu)(\si^{-1}(\chi)),$$
completing the proof.

\end{proof}

\subsubsection{Action on $\md$}
Let $\si\in\Gamma$. Observe that $\gsi^{-1}\si(\open)$ is an open $B$-orbit in
$\Xokb$. Thus $\gsi^{-1}\si(\open)=\open,$ and the map 
$$\md\rightarrow\md,\quad D\mapsto \si\cdot D=\gsii\si(D)$$
is a bijection. The elements of $\md$ being $B$-invariant divisors, this map
does not depend on the choice of $\gsi$. As in the proof of Proposition
\ref{action}, one checks that this defines an action of $\Gamma$ on $\md$.
Moreover, one has the following :
\begin{prop}\label{equivar}
The natural map 
$$\md\rw \mv^B,\quad D\mapsto \nu_D$$
is $\Gamma$-equivariant.
\end{prop}
\begin{proof}
Indeed, if a valuation $\nu$ on $\ka$ has a center $Y$ on $\Xokb$ and if
$\varphi$ is an automorphism of $\Xokb$, then the valuation 
$$\varphi(\nu) :\ka\rw \mq,\quad f\mapsto \nu(\varphi^{-1}(f))$$
is centered on $\varphi(Y)$.
\end{proof}
\subsection{$k$-forms of embeddings}\label{kformspheric}
We assume in this section that $k$ is a perfect field. We keep the notations
from Section \ref{galoisaction}. If $X$ is an embedding of $\Xokb$, we call a
$\textbf{$k$-form}$ of $X$ an embedding of $X_0$ isomorphic to $X$ after
extending scalars to $\kb$. We obtain in this section a criterion for an
embedding $X$ of $\Xokb$ to admit a $k$-form, in terms of its associated
colored fan.
\begin{prop}
If a $k$-form of $X$ exists, then it is unique up to isomorphism.
\end{prop}
\begin{proof}
Take $Y$ and $Z$ two $k$-forms of $X$. There exists an isomorphism of
embeddings 
$$f : Y_{\kb}\rw Z_{\kb}$$
because they are both isomorphic to $X$. But a morphism between two embeddings,
if exists, is unique. This shows that $f$ is unchanged under twisting by
$\Gamma$. In other words, $f$ is defined over $k$.
\end{proof}

The following proposition gives a characterization of embeddings of $\Xokb$
admitting a $k$-form :
\begin{prop}\label{char}
The embedding $X$ admits a $k$-form if and only if it satisfies the following
two conditions :
\begin{enumerate} [(i)]
\item The semi-linear action of $\Gamma$ on $\Xokb$ extends to $X$.
\item $X$ is covered by $\Gamma$-stable affine open sets.
\end{enumerate}
\end{prop}
\begin{rem}\label{quasip}
Condition $(ii)$ is automatically satisfied if $X$ is quasi-projective or
covered by $\Gamma$-stable quasi-projective subsets. By a result of Sumihiro (\cite{Sum}), this is the
case if there is only one closed orbit of $\Gkb$ on $X$, or, in other words, if
the embedding is simple.
\end{rem}
\begin{proof}
These two conditions are satisfied exactly when $X$ admits a $k$-form as a
variety. Because the semi-linear action of $\Gamma$ on $\Xokb$ is
$\Gkb$-semi-linear in the sense that 
$$\forall \si\in\Gamma,\quad \forall g\in G(\kb),\quad \forall x\in X_0(\kb),\quad \si(gx)=\si(g)\si(x),$$
the $k$-form is naturally an embedding of $X_0$.
\end{proof}
Condition $(i)$ in this proposition can be made very explicit in terms of the
colored fan associated to $X$.
\begin{thm}\label{semi}
The $\Gkb$-semi-linear action of $\Gamma$ on $\Xokb$ extends to $X$ if and only
if the colored fan of $X$ is $\Gamma$-stable. In that case, for every
$\Gkb$-orbit $\omega$ on $X$ and every $\si\in\Gamma$ one has 
$$\si(\cat_{\omega})=\cat_{\si(\omega)}\text{ and }\si(\mf_{\omega})=\mf_{\si(\omega)}.$$
\end{thm}
\begin{rem}
We say that a colored fan $\me$ is $\textbf{$\Gamma$-stable}$ if for every
colored cone $(\cat,\mf)\in\me$, the colored cone $(\si(\cat),\si(\mf))$ still
belongs to $\me$.
\end{rem}

\begin{proof}
Assume $(i)$ and let $\omega$ be a $\Gkb$-orbit on $X$. Fix $\si\in \Gamma$.
Observe that $\si(\omega)$ is also a $\Gkb$-orbit on $X$. By mapping a divisor
$D$ to $\si\cdot D$, one gets bijections 
$$\mb_{\omega}\rw \mb_{\si(\omega)}\text{ and } \mf_{\omega} \rw \mf_{\si(\omega)}.$$
Thus, 
$$(\si(\cat_{\omega}),\si(\mf_{\omega}))=(\cat_{\si(\omega)},\mf_{\si(\omega)})$$
because the map $\rho$ is $\Gamma$-equivariant (Proposition \ref{equiva}).

Assume now that the colored fan of $X$ is $\Gamma$-stable. Let $\omega$ be a
$\Gkb$-orbit on $X$. Fix $\si\in\Gamma$, and denote by $\omega'$ the
$\Gkb$-orbit on $X$ satisfying 
$$(\si(\cat_{\omega}),\si(\mf_{\omega}))=(\cat_{\omega'},\mf_{\omega'}).$$
\begin{lemma}
The automorphism $\si$ of $\mv^B$ sends the set $\{\nu_D, D\in\mb_{\omega}\}$
onto the set $\{\nu_D,D\in\mb_{\omega'}\}.$
\end{lemma}
\begin{proof}
Let $D\in \mb_{\omega}$. If $D$ is not $\Gkb$-stable, then $D\in \mf_{\omega}$.
In this case $\si\cdot D\in\mf_{\omega'}$, and thus Proposition \ref{equivar}
shows that 
 $$\si(\nu_D)\in \{\nu_{D'},\quad D'\in\mb_{\omega'}\}.$$
Assume now that $D$ is $\Gkb$-stable. We know from Lemma $2.4$ \cite{Kn} that
$\mq^+\rho(\nu_D)$ is an extremal ray of $\cat_{\omega}$ which contains no
element of $\rho(\mf_{\omega})$. The map $\rho$ being equivariant, this proves
that $\mq^+\rho(\si(\nu_D))$ is an extremal ray of $\cat_{\omega'}$ which
contains no element of $\rho(\mf_{\omega'})$. Using Lemma $2.4$ \cite{Kn} again
and the injectivity of $\rho$ on $\mv^G$, we get that $\si(\nu_D)=\nu_{D'}$,
for some $\Gkb$-stable divisor $D'$ in $\mb_{\omega'}$. So far we have proved
that $\si$ sends the set $\{\nu_D, D\in\mb_{\omega}\}$ into the set
$\{\nu_D,D\in\mb_{\omega'}\}$. Using this result for $\si^{-1}$ and
$\omega'$ instead of $\omega$, one obtains the lemma.
\end{proof}
Using the description of $\kb[X_{\omega,B}]$ and $\kb[X_{\omega',B}]$ given in
Section \ref{warmup} and the previous lemma, we get that the morphism 
$$\open\rw\open,\quad x\mapsto \gsii \si(x)$$
extends to a morphism 
$$X_{\omega,B}\rw X_{\omega',B}.$$
Let $U$ be the largest open subset of $X$ on which this morphism extends. The
action of $\Gamma$ on $\Xokb$ being $\Gkb$-semi-linear, $U$ is $\Gkb$-stable.
But $X$ is covered by the $\Gkb$-translates of $X_{\omega,B}$, $\omega$ being a
$\Gkb$-orbit on $X$. We conclude that $U=X$, completing the proof of the
theorem.
\end{proof}

Assuming that the condition $(i)$ of Proposition \ref{char} holds, we now make
$(ii)$ explicit.
\begin{thm}\label{criterion2}
Let $X$ be an embedding of $\Xokb$ with $\Gamma$-stable colored fan $\me_X$.
Then $X$ admits a $k$-form if and only if for every colored cone
$(\cat,\mf)\in\me_X$, the colored fan consisting of the cones
$(\si(\cat),\si(\mf))_{\si\in\Gamma}$ and their faces is quasi-projective.
\end{thm}
\begin{proof}
The condition given in the proposition is equivalent to the following : for
every $\Gkb$-orbit $\omega$ on $X$, the open subset 
$$\bigcup_{\si\in\Gamma}X_{\si(\omega),\Gkb}$$
is quasi-projective. This set being $\Gamma$-stable, if this condition is
fulfilled then $X$ admits a $k$-form (see Remark \ref{quasip}). To prove the
converse statement, one can clearly replace $X$ by 
$$\bigcup_{\si\in\Gamma}X_{\si(\omega),\Gkb},$$
and thus suppose that maximal cones in $\me_X$ form a single orbit under the
action of $\Gamma$. But maximal cones correspond to closed orbits, so using
Theorem \ref{semi}, one deduces that closed orbits are permuted by $\Gamma$.
Because $X$ admits a $k$-form, there exists an affine open subset $U$ of $X$
meeting every closed $\Gkb$-orbit on $X$. We conclude by Proposition \ref{Su}.
\end{proof}

\subsection{Applications}\label{appli}

\subsubsection{Some situations where $(i)\Rightarrow (ii)$}
We keep notations from Section \ref{galoisaction}, and denote by $X$ an
embedding of $\Xokb$. In Proposition \ref{char} we introduced two conditions
called $(i)$ and $(ii)$ on $X$ which were reformulated in terms of the colored
fan $\me_X$ in Theorems \ref{semi} and \ref{criterion2}. In this section, we
prove that $(i)\Rightarrow (ii)$ under some additional assumptions on $G$,
$X_0$ or $X$..
\begin{prop}\label{nice}
In each of the following situations, one has $(i)\Rightarrow (ii)$ :
\begin{enumerate}
\item $X_0$ is split.
\item $X_0$ is of rank $1$.
\item $X_0$ is horospherical and of rank $2$.
\item $X_0$ is horospherical, and $G$ is split by a quadratic extension $K$ of
$k$.
\item $X$ has no colors, and $X_0$ is of rank $2$.
\item $X$ has no colors, and $G$ is split by a quadratic extension $K$ of
$k$.
\end{enumerate}
\end{prop}
\begin{proof}
We suppose that condition $(i)$ is satified, and prove that $(ii)$ holds. Let
us first consider situation $1.$ We fix a split Borel subgroup $B$ of $G$ with
$Bx_0$ open in $\Xokb$. One can choose $\gsi=1$ for every $\si\in\Gamma$.
Looking at the very definition of the action of $\Gamma$ on $\mx$ and using the
fact that $B$ is split, one deduces that $\Gamma$ acts trivially on $\mx$. Let
$\omega$ be a $\Gkb$-orbit on $X$. For every $\si\in\Gamma$ one has 
 $$(\si(\cat_{\omega}),\si(\mf_{\omega}))=(\cat_{\omega},\mf_{\omega}).$$
But the colored fan consisting of $(\cat_{\omega},\mf_{\omega})$ and its faces
is quasi-projective, so condition $(ii)$ is satisfied.

We now turn to situation $2.$ One can check using Proposition \ref{kazi} and
the fact that $V$ is of dimension $1$ that $X$ is automatically
quasi-projective. Condition $(ii)$ is thus fulfilled.

In the remaining situations, observe that for every closed orbit $\omega$ of
$\Gkb$ on $X$ the cone $\cat_{\omega}$ is contained in $\rho(\mv^G)$ (this is
obvious if $X$ has no colors; if $X_0$ is horospherical use Proposition
\ref{horosph}), so that the collection of cones 
$$\{\cat_{\omega}, \text{ $\omega$ is an orbit of }\Gkb\text{ on }X\}$$
is a fan. Moreover, this fan is quasi-projective if and only if the colored fan
$\me_X$ is.

Every $2$-dimensional fan is quasi-projective, so condition $(ii)$ is satisfied
in situations $3$ and $5.$

If $G$ is split by a quadratic extension $K$ of $k$, then the Galois group
$\Gamma$ acts through a quotient of order $2$ on $V$, and thus for every orbit
$\omega$ of $\Gkb$ on $X$, the fan consisting of the cones
$(\si(\cat_{\omega}))_{\si\in\Gamma}$ and their faces has only one or two
maximal cones. By Remark \ref{deuxprime}, this fan is automatically
quasi-projective. We thus see that condition $(ii)$ is also satisfied in
situations $4$ and $6.$

\end{proof}
\begin{rem}
Situations $1,2,3$ and $4$ don't depend on $X$. This means that in these
situations, the embeddings of $X_0$ are classified by $\Gamma$-stable colored
fans. In the split case, the colored fan $\me_X$ is $\Gamma$-stable if and only
if for every $\Gkb$-orbit $\omega$ on $X$, the set of colors $\mf_{\omega}$ is
$\Gamma$-stable.
\end{rem}

\subsubsection{A spherical embedding with no $k$-form}
In this section $k=\mathbb{R}$. Thus $\kb=\mc$ and $\Gamma=\mz/2\mz$. We let
$\si$ be the non trivial element of $\Gamma$. We construct here a connected
reductive group $G$ over $\mr$, a spherical homogeneous space $X_0$ of rank $2$
under the action of $G$ and an embedding $X$ of $X_{0,\mc}$ whose colored fan
is $\Gamma$-stable, but which admits no $\mr$-form.
\subsubsection{The group $G$}
We denote by $E$ a $3$-dimensional $\mc$-vector space with an $\mr$-structure
$E_{\mr}$. Let $\varepsilon$ be a bilinear form on $E$ defined over $\mr$ and
of signature $(1,2)$. We denote by $q$ its associated quadratic form. If $g$ is
an element of $GL(E)$, we denote by $g^*$ the inverse of its adjoint. We
consider the semi-linear action of $\Gamma$ on $SL(E)$ given by 
$$\forall g\in SL(E),\quad \si(g)=\overline{g}^*.$$
Here the conjugation is relative to the $\mr$-form $SL(E_{\mr })$ of $SL(E)$.
One checks that $\Gamma$ acts by automorphisms of the group $SL(E)$, so that
this action corresponds to a $\mr$-form $G$ of $SL(E)$, which is isomorphic to $SU(2,1)$. Fix an isotropic line
$l$ in $E$ defined over $\mr$, and denote by $B$ the stabilizer of the complete flag 
$$l\subset l^{\perp}$$
in $E$. Then $B$ is a $\Gamma$-stable Borel subgroup of $SL(E)$. Thus $G$ is
quasi-split, and we can choose $g_{\si}=1$ in what follows. The character
lattice of $B$ is given by 
$$\mx(B)=(\mz \chi_1\oplus \mz \chi_2\oplus \mz \chi_3)/\mz
(\chi_1+\chi_2+\chi_3),$$ where the group $B$ acts through the character
$\chi_{1}$ on $l$,
 $\chi_2$ on $l^{\perp}/l$ and
 $\chi_3$ on $E/l^{\perp}$.

We denote by $V$ the dual $\op{Hom}_{\mz}(\mx(B),\mq)$. We thus have 
$$V=\{r_1\mu_1+r_2\mu_2+r_3\mu_3,\quad (r_1,r_2,r_3)\in\mq^3\text{ and }\quad r_1+r_2+r_3=0\}$$
where $\mz \mu_1\oplus \mz \mu_2\oplus \mz \mu_3$ is the dual lattice of $\mz
\chi_1\oplus \mz \chi_2\oplus \mz \chi_3$.
\subsubsection{The homogeneous space $X_0$}
Consider the following affine variety over $\mc$ :
$$Y=\{(p,\plan),\quad p\in E,\quad \plan\subset E\text{ of dimension }2,\quad p\notin \plan\}.$$
The group $SL(E)$ acts naturally and transitively on $Y$. Fix a point $p_0\in
E_{\mr}\setminus\{l\}$ with $q(p_0)=1$. We note $\plan_0=p_0^{\perp}$ and
$x_0=(p_0,\plan_0)$.
\begin{rem}\label{sl2}
The stabilizer of $x_0$ in $SL(E)$ is 
$$\{g\in G(\mc),\quad g(p_0)=p_0\text{ and }g(\plan_0)=\plan_0\},$$
and is therefore isomorphic to $SL_2$. The homogeneous space $Y$ is thus
isomorphic to $SL_3/SL_2$.
\end{rem}
We let the group $\Gamma$ act semi-linearly on $Y$ by 
$$\si(p,\plan)=(p',\overline{p}^{\perp})$$ where $p'\in \overline{\plan}^{\perp}$
satisfies $\varepsilon(\overline{p},p')=1$. We call $X_0$ the $\mr$-form of
this variety corresponding to this semi-linear action. There is no problem to
perform the quotient because $Y$ is an affine variety. Observe that $x_0\in
X_0(\mr)$.
\begin{prop}\label{sph}
The homogeneous space $X_0$ is spherical. More precisely, $B x_0$ is open in
$X_0(\mc)$.
\end{prop}
\begin{proof}
As one can check, the subset 
$$\open:=\{(p,\plan)\in X_0(\mc),\quad p\notin l^{\perp}\text{ and } l\nsubseteq \plan\}$$
of $X_0(\mc)$ is an orbit of $B$, and is open in $X_0(\mc)$. Moreover $x_0\in\open$.
\end{proof}
\subsubsection{Data attached to $X_0$}
We define 
$$D_1:=\{(p,\plan)\in X_0(\mc),\quad p\in l^{\perp}\}\text{ and }D_2:=\{(p,\plan)\in X_0(\mc),\quad l\subseteq \plan\}.$$
These are two $B$-stable divisors on $X_0(\mc)$ and :
\begin{prop}
The set $\md$ equals $\{D_1,D_2\}$, and $\si$ exchanges $D_1$ and $D_2$.
\end{prop}
\begin{proof}
This is a direct consequence of the description of $\open$ given in the proof
of Proposition \ref{sph} and of the definition of the action of $\Gamma$ on
$\md$.
\end{proof}
Before we continue, we need to specify a particular basis of $E_{\mr}$.
\begin{notation}
We denote by $e_1$ the unique vector in $l$ satisfying
$\varepsilon(e_1,p_0)=1$, by $e_2 $ the unique vector in $l^{\perp}\cap
\plan_0$ with $q(e_2)=1$ and by $e_3$ the vector $p_0-e_1$. The vectors
$e_1,e_2,e_3$ give a basis of $E$ defined on $\mr$ and  we denote by
$e_1^*,e_2^*,e_3^*$ the dual basis.
\end{notation}
With these notations, we have 
$$x_0=(e_1+e_3,\langle e_2,e_3\rangle)\text{ and } l^{\perp}=\langle e_1,e_2\rangle.  $$
We define the following two functions on $X_0(\mc)$ 
$$f_1 : (p,\plan) \mapsto e_3^*(p)\text{ and } f_2 : (p,\plan) \mapsto
\frac{\varphi_{\plan}(e_1)}{\varphi_{\plan}(p)}$$ where $\varphi_{\plan}$ is an
equation of the $2$-plane $\plan$. They are respectively the equations of $D_1$
and $D_2$, and $B$-eigenfunctions of weight $-\chi_3$ and $\chi_1$. We are now
able to prove :
\begin{prop}
The lattice $\mx$ is $\mx(B)$ itself. We have 
$$\rho(\nu_{D_1})=\mu_2-\mu_3,\quad \rho(\nu_{D_2})=\mu_1-\mu_2.$$
The automorphism $\si$ of $V$ is the reflection exchanging $\mu_2-\mu_3$
and $\mu_1-\mu_2$.
\end{prop}
\begin{proof}
The first point follows from the fact that $-\chi_3$ and $\chi_1$ generate
$\mx(B)$. Moreover,
$$\rho(\nu_{D_1})(-\chi_3)=\nu_{D_1}(f_1)=1\text{ and }\rho(\nu_{D_1})(\chi_1)=\nu_{D_1}(f_2)=0,$$
proving that $\rho(\nu_{D_1})=\mu_2-\mu_3$. We compute $\rho(\nu_{D_2})$ in the
same way. For the remaining point, recall that the automorphism $\si$ of $\md$
exchanges $D_1$ and $D_2$, and the map $\rho$ is equivariant.
\end{proof}
\begin{prop}
We have 
$$\rho(\mv^G)=\{r_1\mu_1+r_2\mu_2+r_3\mu_3\in V,\quad r_3\geqslant r_1\}.$$
\end{prop}
\begin{proof}
Denote by $\tau$ the permutation $(1, 2, 3)$ and also by $\tau\in SL(E)$ the
automorphism sending $e_i$ to $e_{\tau(i)}$. Then one checks that
$$\tau(f_1)f_2+\tau^2(f_1)\tau^2(f_2)+f_1\tau(f_2)=1$$
in the field $\ka$. This proves that if $\nu\in\mv$ is $G$-invariant, then 
$$\langle\rho(\nu),\chi_1-\chi_3\rangle\leqslant 0.$$
So far we have proved $"\subseteq"$. For the reverse inclusion, we use the
following fact (see \cite{Kn}) : the dimension of the linear part of the cone
$\rho(\mv^G)$ is equal to the codimension of $\op{Stab}(x_0)$ in its
normalizer. Using Remark \ref{sl2} one sees that this normalizer is given by 
$$\{g\in G(\mc),\quad g(p_0)\in \langle p_0\rangle \text{ and }g(\plan_0)=\plan_0\}.$$
Thus the linear part of $\rho(\mv^G)$ is of dimension $1$, completing the proof
of the proposition.
 \end{proof}

\begin{figure}[!ht]
\centering
\includegraphics[width=8cm,height=4.8cm]{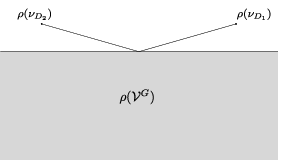}
\caption{The combinatorial data attached to $X_0$}
\end{figure}

\subsubsection{The embedding $X$}
We denote by $\cat$ the cone in $V$ spanned by $\mu_2-\mu_3$ and
$\mu_1-2\mu_2+\mu_3$. Let $\me$ be the colored fan in $V$ with colors in $\md$
whose maximal cones are $(\cat,\{D_1\})$ and $(\si(\cat),\{D_2\})$. The fan
$\me$ is $\Gamma$-stable. It is depicted in Figure $3.$

\begin{figure}[!ht]
\centering
\includegraphics[width=8cm,height=4cm]{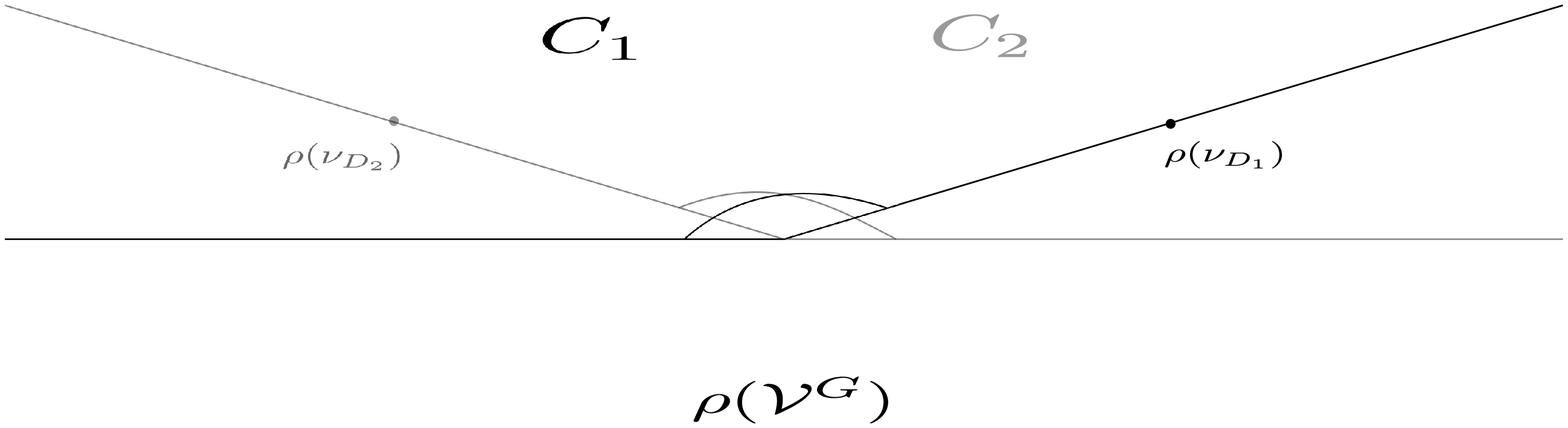}
\caption{The colored fan $\me$.}
\end{figure}

\begin{thm}
The embedding $X=X_{\me}$ of $X_{0,\mc}$ admits no $\mr$-form.
Moreover, $X$ is smooth.
\end{thm}
\begin{proof}
\renewcommand{\qedsymbol}{}
The colored fan $\me$ has two maximal cones that are permuted by $\Gamma$.
Moreover, these maximal cones meet, so the colored fan $\me$ cannot be
quasi-projective. By Theorem \ref{criterion2} this proves the first point. Let
us denote by $\omega$ the closed orbit of $G(\mc)$ on $X$ corresponding to the
colored cone $(\cat,\{D_1\})$. Translating the open subset $X_{\omega,B}$ by
elements of $\Gamma$ and $G(\mc)$ one covers $X$, so in order to prove the
second point, one only has to see that $X_{\omega,B}$ is smooth. Denote by $P$
the stabilizer of $D_2$, or, in other words, the stabilizer of $l$ in $E$. By
Theorem \ref{strucloc}, there exists a Levi subgroup $L$ of $P$ and a closed
$L$-stable subvariety $S$ of $X_{\omega,B}$ containing $x_0$ and such that the
natural map 
$$R_u(P)\times S\rw X_{\omega,B}$$
is a $P$-equivariant isomorphism. By the following lemma, $S$ is isomorphic to
$\mc^3$, and thus smooth. Let $f_3$ be the following function on $X_0(\mc)$ :
$$f_3 : (p,\plan)\mapsto e_2^*(p).$$
\end{proof}
\begin{lemma}\label{lem}
The functions $f_2,f_1 f_2$ and $f_3 f_2$ are algebraically independent in
$\mc[S]$, and 
$$\mc[S]=\mc[f_2,f_1 f_2,f_3 f_2]$$
\end{lemma}
\begin{proof}
 Observe that :
 $$\mc[S]=\mc[X_{\omega,B}]^{R_u(P)}$$ so that the algebra
$\mc[S]$ is the sub-$L$-module of $\mc[Lx_0]$ generated by the $B\cap
L$-eigenfunctions in $\mc(Lx_0)$ of weight $\chi$ satisfying 
$$\langle \mu_2-\mu_3,\chi\rangle\geqslant 0,\quad \langle\mu_1-2\mu_2+\mu_3,\chi\rangle\geqslant 0.$$
In other words, $\chi$ belongs to the monoid generated by $\chi_1$ and
$\chi_1-\chi_3$. The functions $f_2$ and $f_1 f_2$ are $B\cap L$-eigenfunctions
of respective weights $\chi_1$ and $\chi_1-\chi_3$. Moreover, the $L$-module
$\mc[Lx_0]$ is multiplicity-free, because the homogeneous space $Lx_0$ is
spherical. One deduces that
$$\mc[S]=\bigoplus_{(m,n)\in\mathbb{N}^2}\langle L f_2^{m+n}f_1^n\rangle$$
Using the fact that $f_2$ is a $L$-eigenfunction, the following lemma enables
us to conclude.
\renewcommand{\qedsymbol}{}
\end{proof}
\begin{lemma}\label{lemm}
Let $n\in\mathbb{N}$. A basis of the linear span of $L$-translates of $f_1^n$
in $\mc[Lx_0]$ is given by 
 $$f_1^n,f_1^{n-1}f_3,...,f_1 f_3^{n-1},f_3^n.$$
\end{lemma}
\begin{proof} We first prove the case $n=1$. Observe that
$l^{\perp}\subseteq V^*$ is generated by $e_2^*$ and $e_3^*$. Then one easily
checks that $e_2^*,e_3^*$ is a basis of $\langle Le_3^*\rangle$. Moreover, if
there exist $\lambda_1,\lambda_3\in \mc^2$ such that 
$$\lambda_1 f_1+\lambda_3 f_3=0$$
in $\mc[S]$, then the same is true in $\mc[X_{\omega,B}]$. The linear forms
$e_2^*$ and $e_3^*$ are thus linearly dependent in $V^*$, which is a
contradiction. Thus $f_1$ and $f_3$ are linearly independent in $\mc[S]$,
completing the proof in the case $n=1$. The same argument proves that
$f_1^n,f_1^{n-1}f_3,...,f_1 f_3^{n-1},f_3^n$ are linearly independent in
$\mc[S]$. By the case $n=1$, they also generate $\langle Lf_1^n\rangle$.
\end{proof}
\begin{rem}
The embedding $X$ gives an example of a spherical variety containing two points
which do not lie on a common affine open subset. Indeed, there are two closed
orbits of $\Gkb$ on $X$. If the point $y$ belongs to the first orbit and $z$ to
the second, then $y$ and $z$ cannot lie on a common affine open subset, because
of Proposition \ref{Su}. Moreover, in view of Theorem \ref{nice}, this example
has minimal rank.
\end{rem}

\begin{rem}
With a little more work, one can compactify the previous example, and thus
obtains a complete smooth embedding $X$ of $X_0$ with $\Gamma$-stable fan, and
with no $\mr$-form.
\end{rem}

\bibliographystyle{alpha}
\bibliography{ma}

\end{document}